\documentclass{article}
\usepackage{amssymb}
\usepackage{amsfonts}
\usepackage{amsmath}
\usepackage{fullpage}
\usepackage{graphicx}
\usepackage[utf8]{inputenc}
\usepackage{amsthm}
\usepackage{textcomp}

\usepackage{xcolor}

\newtheorem{theorem}{Theorem} 
\newtheorem*{corollary}{Corollary} 
\newtheorem{definition}{Definition} 
\newtheorem{lemma}{Lemma} 

\newcommand{\rP}{\mathrm{P}} 
\newcommand{\rE}{\mathrm{E}} 

\begin{document}

\title{Galton-Watson theta-processes in a varying environment}
\author{Serik Sagitov\thanks{%
Mathematical Sciences, Chalmers University of Technology and the University of Gothenburg. Email address: serik@chalmers.se} \  and Yerakhmet Zhumayev\thanks{%
L.N. Gumilev Eurasian National University, Kazakhstan. Email address: yerakhmet@enu.kz}
}
\maketitle

\begin{abstract}
We consider a special class of Galton-Watson theta-processes in a varying environment fully defined by four parameters, with two of them  $(\theta,r)$ being fixed over time $n$, and the other two $(a_n,c_n)$ characterizing the altering reproduction laws. We establish a sequence of transparent limit theorems for the theta-processes with possibly defective reproduction laws. These results may serve as a stepping stone towards incisive general results for the Galton-Watson processes in a varying environment.
\end{abstract}

\section{Introduction}

The basic version of the Galton-Watson process  (GW-process) was conceived as a stochastic model of the population growth or extinction of a single species of individuals \cite{HJV, KA}. The GW-process $\{Z_{n}\}_{n\ge0}$ unfolds in the discrete time setting, with $Z_n$ standing for the population size at the generation $n$ under the assumption that each individual is replaced by a random number of offspring. It is assumed that the offspring numbers are  independent random variables having the same distribution $\{p(j)\}_{j\ge0}$. 

By allowing the offspring number distribution $\{p_{n}(j)\}_{j\ge0}$ to depend on the generation number $n$, we arrive at the GW-process in a varying environment \cite{Ja}. 
This more flexible model is fully described by a sequence of probability generating functions
$$ f_n (s)=\sum_{j\ge0} p_{n}(j)s^j,\quad 0\le s\le 1,\quad n\ge1.$$ 
Introduce the composition of generating functions
\begin{equation*}
 F_n(s)=f_{1} \circ \ldots \circ f_{n}(s),\quad 0\le s\le 1,\quad n\ge1.
\end{equation*}
Given that the GW-process starts at time zero with a single individual, 
we get
$$\rE(s^{Z_n})=F_n(s),\quad \rP(Z_n=0)=F_{n}(0).$$
The state 0 of the GW-process is absorbing and the extinction probability for the modeled population is determined by
$$q=\lim F_{n}(0)$$
(here and throughout, all limits are taken as $n\to\infty$, unless otherwise specified).
In the case of \textit{proper} reproduction laws with $f_n(1)=1$ for all $n\ge1$, we get
\[\rE(Z_n)=F'_n(1)=f'_{1}(1)  \cdots f'_n(1),\quad \rE(Z_n|Z_n>0)=\frac{F'_n(1)}{1-F_{n}(0)}.\]

In \cite{Ke}, the usual ternary classification of the GW-processes into supercritical, critical, and subcritical processes \cite{AN}, was adapted to the framework of the varying environment. Given $0<f'_n(1)<\infty$ for  all $n$, it was shown that under a regularity condition (A) in \cite{Ke}, it makes sense to
distinguish among four classes of the GW-processes in a varying environment: supercritical, asymptotically degenerate, critical, and subcritical processes. In a more recent paper \cite{SLZ} devoted to the Markov theta-branching processes in a varying environment, the quaternary classification of \cite{Ke} was further refined into a quinary classification, which can be adapted to the discrete time setting as follows:
 \begin{itemize}
\item[ ] \textit{supercritical case}: $q<1$ and $\lim\rE(Z_n)=\infty$, 
\item[ ]  \textit{asymptotically degenerate case}: $q<1$ and $\liminf\rE(Z_n)<\infty$,
\item[ ] \textit{critical case}: $q=1$ and $\lim \rE(Z_n|Z_n>0)=\infty$, 
\item[ ] \textit{strictly subcritical case}: $q=1$ and a finite $\lim \rE(Z_n|Z_n>0)$ exists,
\item[ ] \textit{loosely subcritical case}: $q=1$ and  $\lim \rE(Z_n|Z_n>0)$ does not exist.
\end{itemize}

Our paper is build upon the properties of a special parametric family of generating functions \cite{SL} leading to what will be called here the Galton-Watson theta-processes or GW$^{\,\theta}$-processes. The remarkable property of the GW$^{\,\theta}$-processes  in a varying environment is that the generating functions $F_n(s)$ have explicit expressions presented in Section \ref{theta}. An important feature of  the GW$^{\,\theta}$-processes is that they allow for defective reproduction laws. If the generating function $f_{i}(s)$ is \textit{defective}, in that 
$f_{i}(1)<1$, then $F_{n}(1)<1$ for all $n\ge i$. In the defective case \cite{KM, SM}, a single individual, with probability
$1-f_{i}(1)$ may force the entire GW-process to visit to an ancillary absorbing state $\Delta$ by the observation  time $n$ with probability 
$$\rP(Z_n=\Delta)=1-F_{n}(1).$$

In Sections  \ref{r=1} and  \ref{r>1}, we state ten limit theorems for the  GW$^{\,\theta}$-processes in a varying environment. These results are illuminated in Section \ref{exa} by  ten examples describing different growth and extinction patterns under environmental variation.
The proofs are collected in Section \ref{Spr}.

\section{Proper and defective reproduction laws}\label{theta}
\begin{definition}\label{def1}
Consider a sequence $(\theta,r,a_n,c_n)_{n\ge1}$ satisfying one of the following sets of conditions\\

\begin{tabular}{lllll}
\rm{(a)}&$ \theta\in(0,1]$, &$r=1$, \rm{and for} $n\ge1$,&$0<a_n<\infty$, &$c_n>0$, $c_n\ge 1-a_n$,\\
\rm{(b)}&$ \theta\in(0,1]$, &$r>1$, \rm{and for} $n\ge1$, &$0<a_n<1$, &$(1-a_n)r^{-\theta} \le c_n \le (1-a_n)(r-1)^{-\theta}$,   \\
\rm{(c)}&$ \theta\in(-1,0)$, &$r=1$, \rm{and for} $n\ge1$,  &$0<a_n<1$, &$0<c_n\le 1-a_n$, \\
\rm{(d)}&$ \theta\in(-1,0)$, &$r>1$, \rm{and for} $n\ge1$,  &$0<a_n<1$, &$ (1-a_n)(r-1)^{-\theta} \le c_n \le (1-a_n)r^{-\theta} $, \\
\rm{(e)}&$ \theta=0$, &$r=1$, \rm{and for} $n\ge1$,  &$0<a_n<1$, &$0\le c_n< 1$, \\
\rm{(f)}&$ \theta=0$, &$r>1$, \rm{and for} $n\ge1$,  &$0<a_n<1$, &$0\le c_n\le 1$.
\end{tabular}

\

\noindent A GW$^{\,\theta}$-process with parameters $(\theta,r,a_n,c_n)_{n\ge1}$ is a GW-process in a varying environment characterized by  a sequence of probability generating functions $(f_n(s))_{n\ge1}$ defined by
\begin{equation}\label{fn}
 f_n (s)=r- (a_n(r-s)^{-\theta}+c_n)^{-1/\theta},\quad 0\le s< r,\quad f_n (r)=r,
\end{equation}
for $\theta\ne0$, and for $\theta=0$, defined by 
\begin{equation}\label{fn0}
 f_n (s)=r- (r-c_n)^{1-a_n}(r-s)^{a_n},\quad 0\le s\le r.
 \end{equation}
\end{definition}
Definition \ref{def1} is motivated by the Definitions 14.1 and 14.2 in \cite{SL}, which also mentions a trivial case of $\theta=-1$ not included here. Observe that in the setting of varying environment,  the key parameters $ \theta\in(-1,1]$ and $ r\ge1$, stay constant over  time, while the parameters $(a_n,c_n)$ may vary. The case $\theta=r=1$ is the well studied case of the linear-fractional reproduction law.

This section contains two key lemmas. 
Lemma \ref{L1} gives the explicit expressions for the generating functions $F_n(s)$ in terms of  positive constants $A_n$, $C_n$, $D_n=D_n(r)$  defined by
\[ A_0=1 ,\quad A_n=\prod_{i=1}^n a_i,\quad C_n=\sum_{i=1}^n A_{i-1}c_{i},\quad D_n=\prod_{i=1}^n (r-c_{i})^{A_{i-1}-A_{i}}.\]
 Lemmas \ref{L2} presents the asymptotic properties of the constants $A_n$, $C_n$, $D_n$ leading to the limit theorems stated in Sections \ref{r=1} and \ref{r>1}.
 
\begin{lemma}\label{L1}
Consider a GW$^{\,\theta}$-process with parameters $(\theta,r,a_n,c_n)$.
If $\theta\ne0$, then 
\[F_n(s)=r-(A_n(r-s)^{-\theta}+C_n)^{-1/\theta}, \quad 0\le s<r,\quad F_n(r)=r, \quad n\ge1,\]
and if $\theta=0$, then 
\[F_n(s)=r-(r-s)^{A_{n}}D_n, \quad 0\le s\le r,\quad n\ge1.\]
Here,
\begin{description}
\item[ ] (a) for $\theta\in(0,1]$, $r=1$, 
\[ 0<A_n<\infty,\quad C_n>0,\quad C_n\ge 1-A_n,\quad F_n(1)=1,\quad F'_n(1)=A_n^{-1/\theta}, \quad n\ge1,\]
\item[ ] (b) for $\theta\in(0,1]$, $r>1$, 
\[0<A_n<1,\quad (1-A_n)r^{-\theta}\le C_n\le (1-A_n)(r-1)^{-\theta},\quad F_n(1)\le1, \quad n\ge1,\]
with $F_n(1)=1$ if and only if $c_k=(1-a_k)(r-1)^{-\theta}$, $1\le k\le n$, implying $F'_n(1)=A_n$,
\item[ ] (c) for $\theta\in(-1,0)$, $r=1$, 
\[ 0<A_n<\infty,\quad 0<C_n\le 1-A_n,\quad F_n(1)=1-C_n^{-1/\theta},\quad n\ge1,\]
\item[ ] (d) for $\theta\in(-1,0)$, $r>1$,
\[ 0<A_n<1,\quad (1-A_n)(r-1)^{-\theta}\le C_n\le (1-A_n)r^{-\theta},\quad F_n(1)\le1,\quad n\ge1,\]
with $F_n(1)=1$ if and only if $c_k=(1-a_k)(r-1)^{-\theta}$, $1\le k\le n$, implying $F'_n(1)=A_n$.
\item[ ] (e) for $\theta=0$, $r=1$, 
\[0<A_n<1,\quad 0<D_n\le 1,\quad F_n(1)=1,\quad F'_n(1)=\infty,\quad n\ge1,\]
\item[ ] (f) for $\theta=0$, $r>1$,
\[0<A_n<1,\quad (r-1)^{1-A_n}\le D_n\le r^{1-A_n},\quad F_n(1)\le1, \quad n\ge1,\]
with $F_n(1)=1$ if and only if $c_k=1$, $1\le k\le n$, implying $F'_n(1)=A_n$.
\end{description}
\end{lemma}

\begin{lemma}\label{L2} Denote the limits $A=\lim A_n$, $C=\lim C_n$, $D=\lim D_n$, whenever they exist, whether finite or infinite.
\begin{description}
\item[ ] (a) If $\theta\in(0,1]$, $r=1$, then
$C\in[1,\infty]$, and if $C<\infty$, then $A\in[0,\infty]$.
\item[ ] (b) If $\theta\in(0,1]$, $r>1$, then $A\in[0,1)$ and $(1-A)r^{-\theta}\le C\le (1-A)(r-1)^{-\theta}$.
\item[ ] (c) If $\theta\in(-1,0)$, $r=1$, then $A\in[0,1)$ and  $0<C\le 1-A$,
\item[ ] (d) If $\theta\in(-1,0)$, $r>1$, then $A\in[0,1)$ and  $(1-A)(r-1)^{-\theta}\le C\le (1-A)r^{-\theta}$.
\item[ ] (e) If $\theta=0$, $r=1$, then $A\in[0,1)$ and  $D=\prod_{n\ge1} (1-c_n)^{A_{n-1}-A_{n}}$, with $D\in[0,1]$.
\item[ ] (f) If $\theta=0$, $r>1$, then $A\in[0,1)$ and  $D=\prod_{n\ge1} (r-c_n)^{A_{n-1}-A_{n}}$ with $(r-1)^{1-A}\le D\le r^{1-A}$.
\end{description}
\end{lemma}

\section{Limit theorems for the proper GW$^{\,\theta}$-processes}\label{r=1}

Theorems \ref{thm1}, \ref{thm2}, \ref{thm3}, \ref{thm4}, \ref{thm5} deal with the GW$^{\,\theta}$-process in the case 
$\theta\in(0,1]$,  $r=1$, when
 by Lemma \ref{L1},
\[\rE(Z_n)=A_{n}^{-1/\theta},\quad  \rP(Z_n>0)=(A_n+C_n)^{-1/\theta}.\]
Putting $B_n=C_n/A_n$, we obtain
\[\rE(Z_n|Z_n>0)=(1+B_n)^{1/\theta}.\]
These five theorems fully cover the five regimes of reproduction in a varying environment and could be summarized as follows. Let $\theta\in(0,1]$,  $r=1$,
\begin{description}
 \item[ ] given $C<\infty$,  the GW$^{\,\theta}$-process is
 \begin{description}
\item[ ] supercritical if $A_n\to0$, see Theorem \ref{thm1}, 
\item[ ] asymptotically degenerate if $A_n\to A\in(0,\infty)$, see Theorem \ref{thm2},
\item[ ] strictly subcritical if $A_n\to\infty$, see Theorem \ref{thm4},
\end{description}
 \item[ ] given $C=\infty$, the GW$^{\,\theta}$-process is
 \begin{description}
\item[ ] critical if $B_n\to\infty$, see Theorem \ref{thm3},
\item[ ] strictly subcritical if $B_n\to B\in[0,\infty)$, see Theorem \ref{thm4},
\item[ ] loosely subcritical if the $\lim B_n$ does not exist,  see Theorem \ref{thm5}.
\end{description}

\end{description}

This section also includes Theorem  \ref{thm6} addressing the proper case $\theta=0$, $r=1$. Notice that  Theorem  \ref{thm6} deals with the case of infinite mean values, when the above mentioned quinary classification does not apply.

\begin{theorem}\label{thm1}
Let $\theta\in(0,1]$,  $r=1$, and  $C<\infty$. If $A_n\to0$,  then
$q=1-C^{-1/\theta}$
and  $A_{n}^{1/\theta}Z_n$ almost surely converges to a random variable $W$ such that
\[
\rE(e^{-\lambda W})=1-(\lambda^{-\theta}+C)^{-1/\theta},\quad \lambda\ge 0.
\]
\end{theorem}

\begin{theorem}\label{thm2}
Let $\theta\in(0,1]$,  $r=1$, and  $C<\infty$. If  $A_n\to A\in(0,\infty)$, then 
 \[
q=1-(A+C)^{-1/\theta},\qquad  \rE(Z_n)\to A^{-1/\theta},
\]
 and 
$Z_n$ almost surely converges to a random variable $Z_\infty$ such that 
\[\rE(Z_\infty)=A^{-1/\theta}, \quad \rE(s^{Z_\infty}) =1-(A(1-s)^{-\theta} +C)^{-1/\theta},\quad 0\le s\le 1.\]

\end{theorem}

\begin{theorem}\label{thm3}
Let $\theta\in(0,1]$,  $r=1$, and   $C=\infty$. If $B_n\to\infty$, then $q=1$, 
\[ \rP(Z_n >0)\sim C_n^{-1/\theta},  \qquad  \rE(Z_n|Z_n>0)\sim B_n^{1/\theta} ,
  \]
  and with $\lambda_n=\lambda B_n^{-1/\theta}$,
\[
 \rE(e^{-\lambda_nZ_n}|Z_n>0)\to1-(1+\lambda^{-\theta})^{-1/\theta},\quad \lambda\ge0. 
  \]

\end{theorem}

\begin{theorem}\label{thm4}
Let $\theta\in(0,1]$ and  $r=1$. If $A_n\to\infty$ and $B_n\to B\in[0,\infty)$, then $q=1$, 
\[
 \rP(Z_n >0)\sim (1+B)^{-1/\theta}A_n^{-1/\theta},  \qquad  \rE(Z_n|Z_n>0)\to (1+B)^{1/\theta},  
\]
and 
\[
\rE(s^{Z_{n}}|Z_{n}>0)\to 1-((1+B)(1-s)^{-\theta}+B+B^2)^{-1/\theta},\quad 0\le s\le 1.
\]

\end{theorem}

\begin{theorem}\label{thm5}
Let $\theta\in(0,1]$,  $r=1$, and  assume that $\lim B_n$ does not exist. Then $q=1$ and letting
$$B_{k_n}\to B\in[0,\infty]$$
 along a subsequence  $k_n\to\infty$, we get
\begin{description}
\item[ ] (i) if $B=\infty$, then 
\[  \rP(Z_{k_n} >0)\sim C_{k_n}^{-1/\theta},  \qquad  \rE(Z_{k_n}|Z_{k_n}>0)\sim B_{k_n}^{1/\theta},
\]
 and with $\lambda_n=\lambda B_n^{-1/\theta}$,
\[
 \rE(e^{-\lambda_{k_n}Z_{k_n}}|Z_{k_n}>0)\to 1-(1+\lambda^{-\theta})^{-1/\theta},\quad \lambda\ge0, \]
 
\item[ ] (ii) if $B\in[0,\infty)$, then $A_{k_n}\to\infty$, 
\[ \rP(Z_{k_n} >0)\sim (1+B)^{-1/\theta}A_{k_n}^{-1/\theta},  \qquad  \rE(Z_{k_n}|Z_{k_n}>0)\to (1+B)^{1/\theta}, 
\]
and
\[
\rE(s^{Z_{k_n}}|Z_{k_n}>0)\to  1-((1+B)(1-s)^{-\theta}+B+B^2)^{-1/\theta},\quad 0\le s\le 1.
\]
\end{description}

\end{theorem}

\begin{theorem}\label{thm6}
Suppose $\theta=0$ and $r=1$. Then $\rP(Z_n>0)=D_n$, so that $q=1-D$, with $D$ given by Lemma \ref{L2}(e).
Furthermore,
\begin{description}
\item[ ] (i) if  $A=0$ and $D=0$, then $q=1$ and
\[
\rP(A_n\ln Z_n\le x|Z_n>0) \to 1-e^{-x},\quad x\ge 0,
\]
\item[ ] (ii) if  $A=0$ and $D>0$, then $q<1$ and
\[
\rP(A_n\ln Z_n\le x) \to 1-e^{-x}D,\quad x\ge 0,
\]

\item[ ] (iii) if   $A\in(0,1)$ and $D=0$, then $q=1$ and
\[
\rE(s^{Z_n}|Z_n>0) \to 1-(1-s)^{A}, \quad 0\le s\le 1,
\]
\item[ ] (iv) if   $A\in(0,1)$ and $D>0$, then $q<1$ and $Z_n$ almost surely converges to a random variable $Z_\infty$ such that
\[
\rE(Z_\infty)=\infty,\quad  \rE(s^{Z_\infty}) =1-(1-s)^{A}D,\quad 0\le s\le 1.
\]

\end{description}

\end{theorem}

\subsection*{Remarks}
 \begin{enumerate}
\item It is a straightforward exercise to check that the above mentioned regularity condition (A) in \cite{Ke} is valid for  the GW$^{\,\theta}$-process in the case $\theta\in(0,1]$,  $r=1$. 
\item The limiting distribution obtained in Theorem \ref{thm3} coincides with that of \cite{Slack} obtained for the critical GW-processes in a constant environment with a possibly infinite variance for the offspring number. 
\item The statement (ii) Theorem \ref{thm6} is of the Darling-Seneta type limit theorem obtained in \cite{Da} for GW-processes with infinite mean. 
\item Part (iv) of Theorem \ref{thm6} presents the pattern of limit behavior similar to the asymptotically degenerate regime in the case of infinite mean values. The conditions of Theorem \ref{thm6} (iv) hold if and only if 
 \begin{equation}\label{a0}
 \sum_{n\ge1}(1-a_n)<\infty,
\end{equation}
and 
\begin{equation}\label{A1}
 \sum_{n\ge1} (1-a_n)\ln\tfrac{1}{1-c_n}<\infty.
\end{equation}

\end{enumerate}

\section{Limit theorems for the defective GW$^{\,\theta}$-process}\label{r>1}

In the defective case, there are two kinds of absorption times:
\begin{description}
\item[ ] $\tau_0$ the absorption time of the GW$^{\,\theta}$-process at 0,
\item[ ] $\tau_\Delta$ the absorption time of the GW$^{\,\theta}$-process at the state $\Delta$.
\end{description}
Let $\tau=\min(\tau_0,\tau_\Delta)$ be the absorption time of the GW$^{\,\theta}$-process either at 0 or at the state $\Delta$.
Recall that $q=\rP(\tau_0<\infty)$ and denote
\[q_\Delta=\rP(\tau_\Delta<\infty),\quad Q=\rP(\tau<\infty)=q+q_\Delta. \]
Clearly,
\[\rP(\tau\le n)=\rP(\tau_0\le n)+\rP(\tau_\Delta\le  n)=F_n(0)+1-F_n(1),\]
implying
\[\rP(\tau>n)=F_n(1)-F_n(0).\]
Furthermore,
\[\rE(Z_n;\tau_\Delta>n)=F'_n(1),\quad \rE(s^{Z_n};\tau_\Delta>n)= F_n(s),\quad 0\le s\le1,\]
so that
\[
 \rE(Z_n|\tau>n)=\frac{F'_n(1)}{F_n(1)-F_n(0)},\quad \rE(s^{Z_n}|\tau>n)=\frac{F_n(s)-F_n(0) }{F_n(1)-F_n(0)},\quad 0\le s\le1.
\]

Theorems \ref{thm7}-\ref{thm10} present the transparent asymptotical results on  these absorption probabilities and the limit behavior of the GW$^{\,\theta}$-process in the four defective cases. 
Corollaries  of Theorems \ref{thm7}-\ref{thm9} deal with the proper sub-cases, where $\tau=\tau_0$. All three corollaries describe a strictly subcritical case, when $A=0$, and an asymptotically degenerate case, when $A\in(0,1)$.

\begin{theorem}\label{thm7}
Consider the case $\theta\in(0,1]$,  $r>1$.  Then  
\begin{align*}
 q&=r-(Ar^{-\theta}+C)^{-1/\theta},\quad q_\Delta=1-r+(A(r-1)^{-\theta}+C)^{-1/\theta},
\end{align*}
where $A\in[0,1)$ and $(1-A)r^{-\theta}\le C\le (1-A)(r-1)^{-\theta}$.

\begin{description}
\item[ ] (i) If $A=0$, then  
\[q=1-q_\Delta=r-C^{-1/\theta}\in[0,1],\]
so that $Q=1$.
Furthermore,
  \begin{align*}
A_n^{-1}\rP(\tau>n)&\to ((r-1)^{-\theta}-r^{-\theta}) \theta^{-1}C^{-1/\theta-1},\\
\rE(Z_n|\tau>n)&\to \frac{(r-1)^{-\theta-1}}{(r-1)^{-\theta}-r^{-\theta}},\quad
\rE(s^{Z_n}|\tau>n)\to \frac{(r-s)^{-\theta}-r^{-\theta}}{(r-1)^{-\theta}-r^{-\theta}},\quad 0\le s\le1.
\end{align*}

\item[ ] (ii) If $A\in(0,1)$, then  $Q\in[0,1)$,
\[\rE(Z_n;\tau_\Delta>n)\to A(A+C(r-1)^\theta)^{-1/\theta-1},\]
 and
$Z_n$ almost surely converges to a random variable $Z_\infty$taking values in the set $\{\Delta,0,1,2,\ldots\}$, with
\begin{align*}
 \rP(Z_\infty=\Delta)&=1-r+(A(r-1)^{-\theta}+C)^{-1/\theta},\\ \rE(s^{Z_\infty};Z_\infty\ne\Delta)&=r-(A(r-s)^{-\theta}+C)^{-1/\theta},\quad 0\le s\le1.
\end{align*}
 \end{description}

\end{theorem}

\begin{corollary}
 Consider the case $\theta\in(0,1]$,  $r>1$ assuming 
 \begin{equation}\label{e1}
 c_n = (1-a_n)(r-1)^{-\theta},\quad n\ge1,
\end{equation}
so that 
$C=(1-A)(r-1)^{-\theta}$ implying $q_\Delta=0$.
\begin{description}
\item[ ] (i) If $A=0$, then $q=1$ with
\[A_n^{-1}\rP(Z_n>0)\to ((r-1)^{-\theta}-r^{-\theta}) \theta^{-1}(r-1)^{\theta+1}. \]
Furthermore,
  \[
\rE(Z_n|Z_n>0)\to \frac{(r-1)^{-\theta-1}}{\theta((r-1)^{-\theta}-r^{-\theta})},
\quad
\rE(s^{Z_n}|Z_n>0)\to \frac{(r-s)^{-\theta}-r^{-\theta}}{(r-1)^{-\theta}-r^{-\theta}},\quad 0\le s\le1.
\]
\item[ ] (ii) If $A\in(0,1)$, then
\[q=1-r+(Ar^{-\theta}+C)^{-1/\theta},\quad \rE(Z_n)\to A,\]
so that $q\in(0,1)$,
and 
$Z_n$ almost surely converges to a random variable $Z_\infty$ such that
\[\rE(Z_\infty)=A,\quad \rE(s^{Z_\infty}) = r-(A(r-s)^{-\theta}+(1-A)(r-1)^{-\theta})^{-1/\theta},\quad 0\le s\le1.\]
 \end{description}
\end{corollary}

\begin{theorem}\label{thm8}
Consider the case $\theta\in(-1,0)$, $r>1$ and put $\alpha=-1/\theta$, so that $\alpha>1$. Then  
\begin{align*}
 q&=r-(Ar^{1/\alpha}+C)^{\alpha},\quad q_\Delta=1-r+(A(r-1)^{1/\alpha}+C)^{\alpha},
\end{align*}
where $A\in[0,1)$ and $(1-A)(r-1)^{1/\alpha}\le C\le (1-A)r^{1/\alpha}$.

\begin{description}
\item[ ] (i) If $A=0$, then  
\[q=1-q_\Delta=r-C^{\alpha}\in[0,1],\]
so that $Q=1$.
Furthermore,
  \begin{align*}
A_n^{-1}\rP(\tau>n)&\to  \alpha C^{\alpha-1}(r^{1/\alpha}-(r-1)^{1/\alpha}),\\
\rE(Z_n|\tau>n)&\to \frac{(r-1)^{1/\alpha-1}}{r^{1/\alpha}-(r-1)^{1/\alpha}},\quad
\rE(s^{Z_n}|\tau>n)\to \frac{r^{1/\alpha}-(r-s)^{1/\alpha}}{r^{1/\alpha}-(r-1)^{1/\alpha}},\quad 0\le s\le1.
\end{align*}

\item[ ] (ii) If $A\in(0,1)$, then $Q\in[0,1)$,
\[\rE(Z_n;\tau_\Delta>n)\to A(A+C(r-1)^{-1/\alpha})^{\alpha-1},\]
 and
$Z_n$ almost surely converges to a random variable $Z_\infty$ taking values in the set $\{\Delta,0,1,2,\ldots\}$, with
\begin{align*}
  \rP(Z_\infty=\Delta)&=1-r+(A(r-1)^{1/\alpha}+C)^{\alpha},\\ 
 \rE(s^{Z_\infty};Z_\infty\ne\Delta)&=r-(A(r-s)^{1/\alpha}+C)^{\alpha},\quad 0\le s\le1.
\end{align*}
 \end{description}

\end{theorem}

\begin{corollary}
 Consider the case $\theta\in(-1,0)$, $r>1$ assuming \eqref{e1}, so that 
$C=(1-A)(r-1)^{1/\alpha}$ implying $q_\Delta=0$.
\begin{description}
\item[ ] (i) If $A=0$, then $q=1$ with
\[A_n^{-1}\rP(Z_n>0)\to \alpha (r-1)^{1-1/\alpha}(r^{1/\alpha}-(r-1)^{1/\alpha}) . \]
Furthermore,
  \[
\rE(Z_n|Z_n>0)\to \frac{(r-1)^{-\theta-1}}{\theta((r-1)^{-\theta}-r^{-\theta})},
\quad
\rE(s^{Z_n}|Z_n>0)\to \frac{(r-s)^{-\theta}-r^{-\theta}}{(r-1)^{-\theta}-r^{-\theta}},\quad 0\le s\le 1.
\]
\item[ ] (ii) If $A\in(0,1)$, then
\[q=1-r+(Ar^{1/\alpha}+(1-A)(r-1)^{1/\alpha})^{\alpha},\quad \rE(Z_n)\to A,\]
so that $q\in(0,1)$, and 
$Z_n$ almost surely converges to a random variable $Z_\infty$ such that
\[\rE(Z_\infty)=A, \quad \rE(s^{Z_\infty}) = r-(A(r-s)^{1/\alpha}+(1-A)(r-1)^{1/\alpha})^{\alpha},\quad 0\le s\le 1.\]
 \end{description}
\end{corollary}

\begin{theorem}\label{thm9}

Consider the case $\theta=0$, $r>1$  implying
\[q=r-r^{A}D,\quad q_\Delta=1-r+ (r-1)^{A}D,\quad Q=1-(r^{A}- (r-1)^{A})D,\]
where $D$ is given by Lemma \ref{L2}(f).
\begin{description}
\item[ ] (i) If $A=0$, then  $Q=1$,
and 
\[\rP(\tau>n)\sim (\ln r-\ln(r-1))A_nD_n.\]
Moreover,
\begin{align*}
 \rE(Z_n|\tau>n)&\to\frac{(r-1)^{-1}}{\ln r-\ln(r-1)},\quad \rP(s^{Z_n}|\tau>n)\to\frac{\ln r-\ln(r-s)}{\ln r-\ln(r-1)},\quad 0\le s\le 1.
\end{align*}
\item[ ] (ii) If $A\in(0,1)$, then   $Q<1$, 
\begin{align*}
&(r-1)^{1-A}\le D\le r^{1-A},\\
&\rE(Z_n;\tau_\Delta>n)\to A(r-1)^{A-1} D,
\end{align*}
and $Z_n$ almost surely converges to a random variable $Z_\infty$ taking values in the set $\{\Delta,0,1,2,\ldots\}$, with
\begin{align*}
  \rP(Z_\infty=\Delta)&=1-r+ (r-1)^{A}D,\\ 
 \rE(s^{Z_\infty};Z_\infty\ne\Delta)&=r-(r-s)^{A}D,\quad 0\le s\le 1.
\end{align*}
 \end{description}

\end{theorem}

\begin{corollary}
Given $\theta=0$, $r>1$ , assume  $c_n\equiv1$. Then $D=(r-1)^{1-A}$ implying $q_\Delta=0$.
\begin{description}
\item[ ] (i) If $A=0$, then 
$q=1$,
and 
\[\rP(Z_n>0)\sim (\ln r-\ln(r-1))A_nD_n.\]
Moreover,
\begin{align*}
 \rE(Z_n|Z_n>0)&\to\frac{(r-1)^{-1}}{\ln r-\ln(r-1)},\quad \rP(s^{Z_n}|Z_n>0)\to\frac{\ln r-\ln(r-s)}{\ln r-\ln(r-1)},\quad 0\le s\le 1.
\end{align*}
\item[ ] (ii) If $A\in(0,1)$, then 
\[q=r-r^{A}(r-1)^{1-A},\quad \rE(Z_n)\to A,\]
so that $q\in(0,1)$, and 
$Z_n$ almost surely converges to a proper random variable $Z_\infty$, such that 
\[\rE(Z_\infty)=A,\quad \rE(s^{Z_\infty})=r-(r-s)^{A}(r-1)^{1-A},\quad 0\le s\le 1.
\]
\end{description}
\end{corollary}

\begin{theorem}\label{thm10}

In the case $\theta\in(-1,0)$, $r=1$ , put $\alpha=-1/\theta$, so that $\alpha>1$. Then
\[q= 1-(A+C)^{\alpha},\quad q_\Delta= C^{\alpha},\quad Q=1-(A+C)^{\alpha}+C^{\alpha},\]
where $A\in[0,1)$ and $0< C\le 1-A$.
\begin{description}
\item[ ] (i) If $A=0$, then 
$q=1-q_\Delta=1-C^{\alpha}$, $Q=1$,
and 
\[A_n^{-1}\rP(\tau>n)\to \alpha C^{\alpha-1}.\]
Moreover,
\begin{align*}
 \rE(s^{Z_n}|\tau>n)\to1-(1-s)^{1/\alpha},\quad 0\le s\le 1.
\end{align*}
\item[ ] (ii) If $A\in(0,1)$, then $Q<1$, 
\[\rE(Z_n;\tau_\Delta>n)=\infty,\]
 and
$Z_n$ almost surely converges to a random variable $Z_\infty$ taking values in the set $\{\Delta,0,1,2,\ldots\}$, with
\begin{align*}
  \rP(Z_\infty=\Delta)&=C^{\alpha},\\ 
 \rE(s^{Z_\infty};Z_\infty\ne\Delta)&=1-(A(1-s)^{1/\alpha}+C)^\alpha,\quad 0\le s\le 1.
\end{align*}

 \end{description}

\end{theorem}

 \subsection*{Remarks}
 \begin{enumerate}
\item Theorem \ref{thm7}(ii) should be compared to the more general  Theorem 1 in \cite{KM}, which allows the limit $Z_\infty$ to take the value $\infty$ with a positive probability. 
The convergence results for the conditional expectation should be compared to the statements of Theorems 3 and 4 in \cite{KM}.
\item The conditional convergence in distribution stated in Theorem \ref{thm7}(i) should be compared to Theorem 2a ($k=0$)  of \cite{SM} in the more general setting under the assumption of constant environment.
\end{enumerate}

\section{Examples}\label{exa}
The following ten examples illustrate each of the ten theorems of this paper. Observe that given
\begin{equation}\label{cC1}
 c_n=(1-a_n)\sigma,\quad n\ge1,
\end{equation}
for some suitable positive constant $\sigma$, we get $C_n=(1-A_n)\sigma$, $n\ge1$. Similarly, if 
\begin{equation}\label{cC2}
 c_n=(a_n-1)\sigma,\quad n\ge1,
\end{equation}
for some suitable positive constant $\sigma$, then $C_n=(A_n-1)\sigma$, $n\ge1$.
\subsection*{Example 1}
Suppose $\theta\in(0,1]$, $r=1$, and
\begin{equation}\label{aA1}
 a_n=\frac{n}{n+1},\quad A_n=\frac{1}{n+1},\quad n\ge1.
\end{equation}
If \eqref{cC1} holds for some $\sigma\ge1$, then
by Theorem \ref{thm1}, 
$$q=1-\sigma^{-1/\theta},\quad n^{-1/\theta}\rE(Z_n)\to1, $$ 
and $n^{-1/\theta}Z_n\to W$ almost surely, with
\[\rE(e^{-\lambda W})=1-(\lambda^{-\theta}+\sigma)^{-1/\theta},\quad \lambda\ge0.\]

\subsection*{Example 2}
Suppose $\theta\in(0,1]$, $r=1$, and
\begin{equation}\label{aA2}
 a_n=\frac{n(n+3)}{(n+1)(n+2)},\quad A_n=\frac{n+3}{3(n+1)},\quad n\ge1.
\end{equation}
If \eqref{cC1} holds for some $\sigma\ge1$, then
by Theorem \ref{thm2},
 \[
q=1-(\tfrac{3}{1+2\sigma})^{1/\theta},\qquad  \rE(Z_n)\to 3^{1/\theta},
\]
 and 
$Z_n\to Z_\infty$ almost surely, with
\[\rE(Z_\infty)=3^{1/\theta}, \quad \rE(s^{Z_\infty}) =1-3^{1/\theta}(2\sigma+(1-s)^{-\theta})^{-1/\theta},\quad 0\le s\le 1.\]

\subsection*{Example 3}
Suppose $\theta\in(0,1]$ and $r=1$. Let
\[a_1=\tfrac{1}{2},\quad a_{2n}=4,\quad a_{2n+1}=\tfrac{1}{4},\quad c_{2n-1}=1,\quad c_{2n}=2,\quad A_{2n-1}=\tfrac{1}{2},\quad A_{2n}=2,\quad n\ge1.\]
Then $C=\infty$ and $B_n\to\infty$
implying the conditions of Theorem \ref{thm3}. Observe that for this example, $\lim A_n$ does not exist.
\subsection*{Example 4}
Suppose $\theta\in(0,1]$ and $r=1$. Recall that Theorem \ref{thm4} is the only one among Theorems \ref{thm1}-\ref{thm5} which may hold both with $C<\infty$ and $C=\infty$. For this reason, here we present two examples (1) and (2) for each of these two situations.
\begin{description}
\item[ ] (1) Let
\[a_n=\frac{n+1}{n},\quad c_n=\frac{1}{n^2(n+1)},\quad n\ge1,\]
implying
\[A_n=n+1,\quad C_n=\frac{n}{n+1},\quad B_n=\frac{n}{(n+1)^2},\quad n\ge1.\]
In this case, according to Theorem \ref{thm4},
\[\rP(Z_n>0)\sim n^{-1/\theta},\quad \rE(Z_n|Z_n>0)\to1,\]
and 
\[\rE(s^{Z_n}|Z_n>0)\to s,\quad 0\le s\le1.\]

\item[ ] (2) Let
\[a_n=\frac{n+1}{n},\quad A_n=n+1,\quad  n\ge1,\]
and \eqref{cC2} hold for some $\sigma>0$. Then
\[C_n=\sigma n,\quad B_n=\frac{\sigma n}{n+1},\quad n\ge1.\]
In this case, according to Theorem \ref{thm4},
\[\rP(Z_n>0)\sim (1+\sigma)^{-1/\theta}n^{-1/\theta},\quad \rE(Z_n|Z_n>0)\to(1+\sigma)^{1/\theta},\]
and 
\[\rE(s^{Z_n}|Z_n>0)\to1-((1+\sigma)(1-s)^{-\theta}+\sigma+\sigma^2)^{-1/\theta},\quad 0\le s\le 1.\]

\end{description}

\subsection*{Example 5}
Suppose $\theta\in(0,1]$ and $r=1$. 
Let
\begin{equation*}
  a_n =
    \begin{cases}
      n & \text{for $n = 2^k -1$, $k \ge1$},\\
      1/(n-1) & \text{for $n = 2^k $, $k \ge1$},\\
      1 & \text{otherwise},
    \end{cases}       
\qquad   A_n =
    \begin{cases}
      n & \text{for $n = 2^k -1$, $k \ge1$},\\
      1 & \text{otherwise}.
    \end{cases}       
\end{equation*}
Taking
\begin{equation*}
  c_n =
    \begin{cases}
      1 & \text{for $n = 2^k $, $k > 1$},\\
      1/n^2 & \text{otherwise},
    \end{cases}       
\end{equation*}
we get
\begin{equation*}
C_n = \sum_{k:2\le 2^k\le n} (2^k-1-2^{-2k}) + \sum_{k=1}^n k^{-2},\quad n\ge1,
\end{equation*}
implying
$C_{k_n} \sim 2^{n+1}$, 
provided $2^{n}-1\le k_n< 2^{n+1}-1$.
Thus, by Theorem 5, for $k_n=2^n$, $\lambda_n=\lambda(2n)^{-1/\theta}$,
\begin{align*}
&\rP(Z_{k_n}>0) \sim (2k_n)^{-1/\theta},\\
 &\rE(e^{-\lambda_{k_n}Z_{k_n}}|Z_{k_n}>0)\to 1-(1+\lambda^{-\theta})^{-1/\theta},\quad \lambda\ge0, 
\end{align*}
and on the other hand, for $k_n=2^n-1$,
\begin{align*}
&\rP(Z_{k_n}>0) \sim (3k_n)^{-1/\theta},\\
 &\rE(s^{Z_{k_n}}|Z_{k_n}>0)\to 1- (3(1-s)^{-\theta}+6)^{-1/\theta},\quad 0\le s\le 1. 
\end{align*}

\subsection*{Example 6}
Suppose $\theta=0$, $r=1$, and assume $c_n=1-e^{-n^{\sigma}}$, $ -\infty<\sigma<\infty$, $n\ge1$,
yielding
\[D_n=\exp\Big(-\sum_{i=1}^n i^\sigma (A_{i-1}-A_i)\Big),\quad n\ge1.\]
Notice that \eqref{aA1} implies $A=0$ and
\[D_n=\exp\Big(-\sum_{i=1}^n\frac{ i^{\sigma-1}}{i+1} \Big),\quad n\ge1,\]
on the other hand, 
\eqref{aA2} implies $A=1/3$ and
\[D_n=\exp\Big(-\sum_{i=1}^n\frac{ 2i^{\sigma-1}}{3(i+1)}  \Big),\quad n\ge1.\]
\begin{description}
\item[ ] (i) If \eqref{aA1} holds and $\sigma\ge1$, then 
\[A_n\sim n^{-1},\quad D_n=\exp\Big(-\sum_{i=1}^n\frac{ i^{\sigma-1}}{i+1} \Big)\to 0,\]
so that the conditions of Theorem \ref{thm6}(i) are satisfied.

\item[ ] (ii) If \eqref{aA1} holds and $\sigma<1$, then 
\[A_n\sim n^{-1},\quad D=\exp\Big(-\sum_{i=1}^\infty\frac{ i^{\sigma-1}}{i+1} \Big),\]
so that the conditions of Theorem \ref{thm6}(ii) are satisfied.

\item[ ] (iii) If \eqref{aA2} holds and $\sigma\ge1$, then 
\[A=1/3,\quad D_n=\exp\Big(-\sum_{i=1}^n\frac{ 2i^{\sigma-1}}{3(i+1)}  \Big)\to 0,\]
so that the conditions of Theorem \ref{thm6}(iii) are satisfied.

\item[ ] (iv) If \eqref{aA2} holds and $\sigma<1$, then 
\[A=1/3,\quad D=\exp\Big(-\sum_{i=1}^\infty\frac{ 2i^{\sigma-1}}{3(i+1)}  \Big),\]
so that the conditions of Theorem \ref{thm6}(iv) are satisfied.
\end{description}

\subsection*{Example 7}
Suppose $\theta\in(0,1]$, $r>1$  assuming \eqref{cC1} with 
$r^{-\theta}\le \sigma \le (r-1)^{-\theta}$.
\begin{description}
\item[ ] (i) If \eqref{aA1},
then the conditions of Theorem \ref{thm7}(i) hold with $A_n\sim n^{-1}$ and $C=\sigma$.

\item[ ] (ii) If \eqref{aA2},
then
the conditions of  Theorem \ref{thm7}(ii) hold with $A=1/3$ and $C=2\sigma /3$.
\end{description}

\subsection*{Example 8}
Suppose $\theta\in(-1,0)$, $r>1$
assuming \eqref{cC1} with 
$r-1\le \sigma^\alpha \le r$,
where $\alpha=-1/\theta$.

\begin{description}
\item[ ] (i) If \eqref{aA1},
then the conditions of Theorem \ref{thm8}(i) hold with $A_n\sim n^{-1}$ and $C=\sigma$.

\item[ ] (ii) If \eqref{aA2},
then
the conditions of  Theorem \ref{thm8}(ii) hold with $A=1/3$ and $C=2\sigma /3$.
\end{description}

\subsection*{Example 9}
Suppose $\theta=0$ and $r>1$ and assume
\[ c_n = \sigma,\quad 0\le\sigma\le1,\quad n\ge1,  \]
which implies
\[ D_n = (r-\sigma)^{1-A_n},\quad n\ge1. \]
\begin{description}
\item[ ] (i)  If  \eqref{aA1}, then by Theorem 9(i), we get in particular,
\[ \rP(\tau>n)\sim \gamma n^{-1},\quad \gamma=(r-\sigma)\ln\tfrac{r}{r-1} . \]
\item[ ] (ii)   If  \eqref{aA1}, then by Theorem 9, we get in particular,
\[q=r-r^{1/3}(r-\sigma)^{2/3},\quad q_\Delta=1-r+ (r-1)^{1/3}(r-\sigma)^{2/3},\quad Q=1-(r^{1/3}- (r-1)^{1/3})(r-\sigma)^{2/3}.\]
\end{description}

\subsection*{Example 10}
Suppose $\theta\in(-1,0)$, $r=1$. Put  $\alpha = -1/\theta$ and assume \eqref{cC1} with 
$0 < \sigma \le 1$.
\begin{description}
\item[ ] (i)   If \eqref{aA1},
then by Theorem 10(i), we get in particular, $q_\Delta=\sigma^\alpha$ and 
\[\rP(\tau>n)\sim \alpha\sigma^{\alpha-1}n^{-1}.\]

\item[ ] (ii)  If \eqref{aA2}, then by Theorem 10(ii), we get in particular, $Q = 1 - (\frac{1}{3} + \frac{2\sigma}{3})^\alpha + {\frac{2\sigma}{3}}^\alpha $.
\end{description}

\section{Proofs}\label{Spr}

In this section we sketch the proofs of lemmas and theorems of this paper.
The corollaries to Theorems \ref{thm7}-\ref{thm9} are easily obtained from the corresponding theorems.

\subsection*{Proof of Lemma \ref{L1}}
 
Relations \eqref{fn} and \eqref{fn0} 
 imply respectively 
\[(r-f_k\circ f_{k+1} (s))^{-\theta}=a_k(r-f_{k+1}  (s))^{-\theta}+c_k=a_ka_{k+1} (r-s)^{-\theta}+c_k+a_kc_{k+1} ,\]
and
\[r-f_k\circ f_{k+1} (s)=(r-c_k)^{1-a_k}(r-f_{k+1} (s))^{a_k}=(r-c_k)^{1-a_k}(r-c_{k+1})^{(1-a_{k+1})a_k}(r-s)^{a_ka_{k+1}},\]
entailing the main claims of Lemma \ref{L1}. The parts (a)-(f) follow from the respective restrictions (a)-(f) on $(a_n,c_n)$ stated in the Definition 1.

\subsection*{Proof of Lemma \ref{L2}}

\

(a) In the case $\theta\in(0,1]$, $r=1$, the claim follows from the existence of $\lim C_n$ and $\lim (A_n+C_n)$, which in turn,  follows from monotonicity of the two sequences. To see that $A_n+C_n\le A_{n+1}+C_{n+1}$, it suffices to observe that
\[A_n-A_{n+1}=A_n(1-a_{n+1})\le A_nc_{n+1}=C_{n+1}-C_n.\]
The second part of Lemma \ref{L2} is a direct implication of the definition of $C_n$.

  (b)-(f). The rest of the stated results follows immediately from the restrictions  (b)-(f) imposed on $(a_n,c_n)$ in Definition 1.

\subsection*{Church-Lindvall condition for the GW$^{\,\theta}$-process}
In \cite{Lin} it was shown for the GW-processes in a varying environment that  the almost surely convergence $Z_n\stackrel{a.s.}{\to} Z_\infty$ holds with $\rP(0<Z_\infty<\infty)>0$ if and only if the following condition holds:
\begin{equation}\label{CL}
 \sum_{n\ge1} (1-p_n(1))<\infty.
\end{equation}
Relation \eqref{CL} is equivalent to
\begin{equation}\label{CL1}
 \prod_{n\ge n_0} p_n(1)>0,
\end{equation}
for some $n_0\ge1$.
For the GW$^{\,\theta}$-process, the equality $p_n(1)=f'_n(0)$ implies 
\begin{equation}\label{pn1}
p_n(1)=a_n(a_n+c_nr^{\theta})^{-1/\theta-1},
\end{equation}
for $\theta\ne0$, and for $\theta=0$,
\begin{equation}\label{pn2}
p_n(1)=a_n(1-c_nr^{-1})^{1-a_n}.
\end{equation}

\begin{lemma}\label{LCL1}
In the case $\theta\in(0,1]$ and $r=1$, relation \eqref{CL} holds if and only if 
 \begin{equation}\label{ant}
A_n\to A\in(0,\infty)\end{equation}
 and
 \begin{equation}\label{cnt}
\sum_{n\ge1} c_n<\infty.
\end{equation}

\end{lemma}

\begin{proof}
I view of \eqref{pn1}, we have
 \[  \prod_{i=1}^n p_i(1)=A_n G_n^{-1/\theta-1},\quad G_n:=\prod_{i=1}^n (a_i+c_i).\]
Since $a_n+c_n\ge1$, we have 
\[\lim G_n=G\in[1,\infty].\]
If $G=\infty$, then \eqref{CL1} is not valid, implying that \eqref{CL} is equivalent to  \eqref{ant} plus $G<\infty$. It remains to verify that under \eqref{ant}, the inequality $G<\infty$ is equivalent to  \eqref{cnt}. Suppose \eqref{ant} holds, and observe that in this case, $G<\infty$ is equivalent to
\[\prod_{n\ge1} (1+c_n/a_n)<\infty,\]
which is true if and only if
\[\sum_{n\ge1} c_n/a_n<\infty.\]
Since under \eqref{ant}, $a_n\to1$, the latter condition is equivalent to  \eqref{cnt}. 
\end{proof}

\begin{lemma}\label{LCL2}
 In the case $\theta=0$ and $r=1$,  relation \eqref{CL} holds if and only if  
 $A\in(0,1)$ and $D\in(0,1)$.
\end{lemma}

\begin{proof}
In view of \eqref{pn2}, we have
 \[  \prod_{n\ge1} p_n(1)=A \prod_{n\ge1}(1-c_n)^{1-a_n}.\]
 It remains to observe that given $A\in(0,1)$ the relation $D\in(0,1)$  is equivalent to
 \[\prod_{n\ge1} (1-c_n)^{1-a_n}>0.\]
 \end{proof}

\begin{lemma}\label{LCL3}
Assume that $\theta\ne0$ and $r>1$, and consider  $\{\tilde Z_n\}$, a GW-process in a varying environment with the proper probability generating functions
\[ \tilde f_n(s)=\frac{f_n(s)}{f_n(1)}=\frac{r- (a_n(r-s)^{-\theta}+c_n)^{-1/\theta}}{r- (a_n(r-1)^{-\theta}+c_n)^{-1/\theta}}.\]
Relation \eqref{a0} implies
\[ \sum_{n=1}^\infty (1-\tilde p_n(1))<\infty.\]
\end{lemma}

\begin{proof} Assume   $\theta\in(0,1]$ and $r>1$ together with \eqref{a0}. Then $A_n\to A\in(0,1)$, $a_n\to1$, and $c_n\to0$.
We have
\[\tilde p_n(1)=\tilde f'_n(0)=a_nh_n^{-1/\theta-1}k_n^{-1},\]
where
\[h_n=a_n+c_nr^{\theta},\quad k_n=r- (a_n(r-1)^{-\theta}+c_n)^{-1/\theta},\]
are such that $h_n\ge 1$ and $k_n\in(0,1]$. 
The statement follows from the representation
 \[  \prod_{n\ge1} p_n(1)=A H^{-1/\theta-1}K^{-1},\]
 where $H=\prod_{n\ge1} h_n$ and $ K=\prod_{n\ge1} k_n$.
 It is easy to show that \eqref{a0} and $(1-a_n)r^{-\theta} \le c_n \le (1-a_n)(r-1)^{-\theta}$ yield 
 \[\sum_{n\ge1} (h_n-1)\le r^\theta\sum_{n\ge1} c_n\le r^\theta(r-1)^{-\theta}\sum_{n\ge1} (1-a_n)<\infty,\]
implying
 $H\in[1,\infty)$. On the other hand, $K\in(0,1]$, since
 \[\sum_{n\ge1}(1-k_n)<\infty,\]
 which follows from
 \[1-k_n\le (r-1)(a_n+c_n(r-1)^{\theta})^{-1/\theta}-1)\le r(1-(a_n+c_n(r-1)^{\theta})^{1/\theta})\le r\theta^{-1}(1-a_n).\]
 
 In the other case, when   \eqref{a0} holds together with $\theta\in(-1,0)$ and $r>1$ , the lemma is proven similarly. 
\end{proof}

\subsection*{Proof of Theorems \ref{thm1}, \ref{thm2}, \ref{thm3}, \ref{thm4}, \ref{thm5}}

The proofs of these theorems are done using the usual for these kind of results arguments applied to the explicit expressions available for $F_n(s)$. In particular, the following standard formula is a starting point for computing the conditional limit distributions
\begin{equation}\label{Co}
 \rE(s^{Z_n}|Z_n>0) =\frac{\rE(s^{Z_n})-\rP(Z_n=0) }{\rP(Z_n>0)}=1-\frac{1-F_n(s) }{1-F_n(0)}.
\end{equation}
Thus in the case $\theta\in(0,1]$ and $r>1$, Lemma \ref{L1} and \eqref{Co} imply
\[
\rE(s^{Z_n}|Z_n>0) =1-\frac{((1-s)^{-\theta} +B_n)^{-1/\theta}}{(1+B_n)^{-1/\theta}}\to 1-\frac{((1-s)^{-\theta} +B)^{-1/\theta}}{(1+B)^{-1/\theta}},
\]
proving the main statement of Theorem \ref{thm4}. The almost sure convergence stated in Theorem \ref{thm2} follows from Lemma \ref{LCL1} and the earlier cited criterium of  \cite{Lin}. 

\subsection*{Proof of Theorem \ref{thm6}}

Suppose $\theta=0$, $r=1$, in which case $A\in[0,1)$ and $D\in[0,1]$.\\

(i) Suppose $A=D=0$. In this case $q=1-D=1$, and by  \eqref{Co} and Lemma \ref{L1},
\[
\rE(s^{Z_n}|Z_n>0)=1-(1-s)^{A_n}.
\]
Putting here $s_n=\exp{(-\lambda e^{-x/A_n})}$ we get as $n\to\infty$,
\[
\rE(s_n^{Z_n}|Z_n>0)=1-(1-\exp(-\lambda e^{-x/A_n}))^{A_n}=1-\exp(A_n\ln(\lambda e^{-x/A_n}(1+o(1)))\to 1-e^{-x}.
\]
This implies a convergence in distribution
\[(Z_ne^{-x/A_n}|Z_n>0)\stackrel{d}{\to}W(x),\]
where the limit $W(x)$ has a degenerate distribution with
\[\rP(W(x)\le w)=(1-e^{-x})1_{\{0\le w<\infty\}}.\]
In other words,
\[\rP(Z_n\le we^{x/A_n}|Z_n>0)\to(1-e^{-x})1_{\{0\le w<\infty\}}.\]
After taking the logarithm of $Z_n$, we arrive at the statement of Theorem \ref{thm6}(i).\\

(ii) The statement (ii)  follows from Lemma \ref{L1} and relation \eqref{Co} in a similar way as the statement (i). \\

(iii)  If   $A\in(0,1)$ and $D=0$, then $q=1$ and by relation \eqref{Co} and Lemma \ref{L1} 
\[ \rE(s^{Z_n}|Z_n>0) =1-(1-s)^{A_n}\to 1-(1-s)^{A}.\]

(iv) Let $A>0$ and $D>0$. Since $q=1-D$, similarly to the part (iii), we obtain
\[\rE(s^{Z_n})\to1-(1-s)^AD.\]
By Lemma \ref{LCL2}, the convergence in distribution $Z_n\stackrel{d}{\to} Z_\infty$ can be upgraded to the almost surely convergence $Z_n\stackrel{a.s.}{\to} Z_\infty$.

\subsection*{Proof of Theorems \ref{thm7} and \ref{thm8} }
In this section we prove only Theorem \ref{thm7}. Theorem \ref{thm8} is proven similarly. 

By Lemma \ref{L1}, 
\[F_n(0)= r-(A_nr^{-\theta}+C_n)^{-1/\theta},\quad F_n(1)= r-(A_{n}(r-1)^{-\theta}+C_n)^{-1/\theta}.\]
It follows,
\begin{align*}
 \rP(\tau>n)&= (A_nr^{-\theta}+C_n)^{-1/\theta}-(A_n(r-1)^{-\theta}+C_n)^{-1/\theta},\\
 \rE(Z_n|\tau>n)&=\frac{F'_n(1)}{F_n(1)-F_n(0)}= \frac{\theta^{-1}A_{n}(A_{n}+C_n(r-1)^{\theta})^{-1/\theta-1}}{(A_nr^{-\theta}+C_n)^{-1/\theta}-(A_{n}(r-1)^{-\theta}+C_n)^{-1/\theta}},\\
 \rE(s^{Z_n}|\tau>n)&=\frac{F_n(s)-F_n(0) }{F_n(1)-F_n(0)}= \frac{(A_{n}r^{-\theta}+C_n)^{-1/\theta}-(A_{n}(r-s)^{-\theta}+C_n)^{-1/\theta}}{(A_nr^{-\theta}+C_n)^{-1/\theta}-(A_{n}(r-1)^{-\theta}+C_n)^{-1/\theta}}.
\end{align*}

(i) Assume that $A=0$. Then the sequence of positive numbers
\[V_n=A_n^{-1}(C-C_n)=c_{n+1}+c_{n+2}a_{n+1}+c_{n+3}a_{n+2}a_{n+1}+\ldots\]
satisfies
\[r^{-\theta}\le \liminf V_n\le \limsup V_n\le (r-1)^{-\theta}.\]
For a given $x\in(0,\infty)$, put
\[W_n(x)=A_n^{-1} (C^{-1/\theta}-(A_nx+C_n)^{-1/\theta}).\]
Since
\[W_n(x)=A_n^{-1} (C^{-1/\theta}-(A_n(x-V_n)+C)^{-1/\theta})=\theta^{-1}C^{-1/\theta-1}(x-V_n+o(1)),\]
the representation
\[A_n^{-1}\rP(\tau>n)= W_n((r-1)^{-\theta})-W_n(r^{-\theta})\]
yields the first asymptotic result stated in the part (i) of Theorem \ref{thm7}. The other two asymptotic results follow from the representations
\begin{align*}
 \rE(Z_n|\tau>n)&=    \frac{\theta^{-1}(A_{n}+C_n(r-1)^{\theta})^{-1/\theta-1}}{W_n((r-1)^{-\theta})-W_n(r^{-\theta})},\\
 \rE(s^{Z_n}|\tau>n)&= \frac{W_n((r-s)^{-\theta})-W_n(r^{-\theta})}{W_n((r-1)^{-\theta})-W_n(r^{-\theta})}.
\end{align*}

(ii) The second claim follows from the equality
\[ \rE(s^{Z_n};\tau_\Delta>n)=r-(A_{n}(r-s)^{-\theta}+C_n)^{-1/\theta}.\]

\subsection*{Proof of Theorem \ref{thm9}}
If $\theta=0$ and $r>1$, then by Lemma \ref{L1}
\[\rP(Z_n=0)= r-r^{A_{n}}D_n,\quad \rP(Z_n\ne\Delta)= r-(r-1)^{A_{n}}D_n.\]
It follows,
\[q=r-r^{A}D,\quad q_\Delta=1-r+ (r-1)^{A}D,\quad Q=1-(r^{A}- (r-1)^{A})D.\]

(i) If $A=0$, then clearly
\[q=r-D,\quad q_\Delta=1-r+ D,\quad Q=1,\]
and
\[\rP(\tau>n)= (r^{A_{n}}-(r-1)^{A_{n}})D_n\sim (\ln r-\ln(r-1))A_nD_n.\]
Furthermore,
\begin{align*}
 \rE(Z_n|\tau>n)&=\frac{F'_n(1)}{F_n(1)-F_n(0)}=\frac{A_n(r-1)^{A_{n}-1}}{r^{A_{n}}-(r-1)^{A_{n}}}\to\frac{(r-1)^{-1}}{\ln r-\ln(r-1)},\\
\rP(s^{Z_n}|\tau>n)&=\frac{F_n(s)-F_n(0) }{F_n(1)-F_n(0)}=\frac{r^{A_n}-(r-s)^{A_{n}}}{r^{A_{n}}-(r-1)^{A_{n}}}\to\frac{\ln r-\ln(r-s)}{\ln r-\ln(r-1)}.
\end{align*}

(ii) In the case $A>0$, the main claim is obtained as
\[ \rE(s^{Z_n};\tau_\Delta>n)=r-(r-s)^{A_{n}}D_n\to r-(r-s)^{A}D.\]

\subsection*{Proof of Theorem \ref{thm10}}
If $\theta\in(-1,0)$ and $r=1$, then by Lemmas \ref{L1}, \ref{L2}, 
\[F_n(0)= 1-(A_n+C_n)^{\alpha},\quad F_n(1)= 1-C_n^{\alpha},\]
\[q= 1-(A+C)^{\alpha},\quad q_\Delta= C^{\alpha},\]
where $\alpha=-1/\theta$ and $0< C\le 1-A$.\\

(i) Suppose $A=0$. Then the sequence of positive numbers
$V_n=A_n^{-1}(C-C_n)$
satisfies
\[0\le \liminf V_n\le \limsup V_n\le 1.\]
For a given $x\in(0,\infty)$, put
\[W_n(x)=A_n^{-1} ((A_nx+C_n)^{\alpha}-C^{\alpha}).\]
Since
\[W_n(x)=\alpha C^{\alpha-1}(x-V_n+o(1)),\]
the representation
\[A_n^{-1}\rP(\tau>n)= W_n(1)-W_n(0)\]
yields the first asymptotic result stated in the part (i) of Theorem \ref{thm10}. The other asymptotic result follows from the representation
\begin{align*}
 \rE(s^{Z_n}|\tau>n)&= \frac{W_n(1)-W_n((1-s)^{1/\alpha})}{W_n(1)-W_n(0)}.
\end{align*}

(ii)
The claim (ii) is derived as
\[
\rP(s^{Z_n};\tau>n)=1-(A_n(1-s)^{1/\alpha}+C_n)^\alpha\to 1-(A(1-s)^{1/\alpha}+C)^\alpha.\]

\end{document}